\documentclass[reqno]{amsart}
\usepackage{amsmath, amsthm, amssymb, amstext}

\usepackage{hyperref,xcolor}
\hypersetup{
 pdfborder={0 0 0},
 colorlinks,
}
\usepackage{enumitem}
\newtheorem{theorem}{Theorem}
\newtheorem{remark}[theorem]{Remark}
\newtheorem{lemma}[theorem]{Lemma}

\DeclareMathOperator*{\divergenz}{div}              %

\newcommand{\R}{\mathbb{R}}

\newcommand{\RN}{\mathbb{R}^N}

\newcommand{\Lp}[1]{L^{#1}(\Omega)}

\newcommand{\Wpzero}[1]{W^{1,#1}_0(\Omega)}

\newcommand{\eps}{\varepsilon}

\newcommand{\rand}{\partial\Omega}

\newcommand{\into}{\int_{\Omega}}

\newcommand{\Linf}{L^{\infty}(\Omega)}

\renewcommand{\l}{\left}
\renewcommand{\r}{\right}
\numberwithin{theorem}{section}
\numberwithin{equation}{section}

\title[$L^\infty$-bounds for general singular elliptic equations]{$L^\infty$-bounds for general singular elliptic equations with convection term}

\author[G.\,Marino]{Greta Marino}
\address[G.\,Marino]{Technische Universit\"{a}t Chemnitz, Fakult\"{a}t f\"{u}r Mathematik, Reichenhainer Stra\ss e 41, 09126 Chemnitz, Germany}
\email{greta.marino@mathematik.tu-chemnitz.de}

\author[P.\,Winkert]{Patrick Winkert}
\address[P.\,Winkert]{Technische Universit\"{a}t Berlin, Institut f\"{u}r Mathematik, Stra\ss e des 17.\,Juni 136, 10623 Berlin, Germany}
\email{winkert@math.tu-berlin.de}



\begin{document}

\begin{abstract}
    In this note we present $L^\infty$-results for problems of the form
    \begin{align*}
	\begin{aligned}
	    -\divergenz \mathcal A(x, u, \nabla u)&= \mathcal B(x, u, \nabla u) \qquad && \text{in } \Omega, \\
	    u&> 0 && \text{in } \Omega, \\
	    u&= 0 && \text{on } \rand,
	    \end{aligned}
    \end{align*}
    where the growth condition for the function $\mathcal{B}\colon \Omega \times \R\times \R^N\to \R$ contains both a singular and a convection term. We use ideas from the works of Giacomoni-Schindler-Tak\'{a}\v{c} \cite{Giacomoni-Schindler-Takac-2007} and the authors \cite{Marino-Winkert-2019} to prove the boundedness of weak solutions for such general problem by applying appropriate bootstrap arguments.
\end{abstract}

\maketitle

\section{Introduction and Assumptions}

Let $ \Omega \subset \RN$ be a bounded domain with a Lipschitz boundary $\rand$. In this paper, we are concerned with the following problem
\begin{align}\label{problem}
    \begin{aligned}
	-\divergenz \mathcal A(x, u, \nabla u)&= \mathcal B(x, u, \nabla u) \qquad && \text{in } \Omega, \\
	u&> 0 && \text{in } \Omega, \\
	u&= 0 && \text{on } \rand,
    \end{aligned}
\end{align}
where we assume the subsequent hypotheses:
\begin{enumerate}
    \item[(H)] 
	The functions $ \mathcal A\colon \Omega \times \R \times \RN \to \R^N $ and $ \mathcal B\colon \Omega \times \R \times \RN \to \R $ are supposed to be Carath\'eodory functions such that
	\begin{align*}
	    &\qquad&\text{(H1)}  
	    & \quad | \mathcal A(x, s, \xi)| \le a_1 | \xi |^{p-1}+ a_2 | s |^{q \frac{p-1}{p}}+ a_3,  
	    & & \text{for a.a. } x \in \Omega, \\
	    &&\text{(H2)} 
	    &\quad \mathcal A(x, s, \xi) \cdot \xi \ge a_4 | \xi |^p, 
	    & & \text{for a.a. } x \in \Omega, \\
	    &&\text{(H3)} 
	    &\quad | \mathcal B(x, s, \xi) | \le b_1 | \xi |^{p \frac{q-1}{q}}+ b_2 | s |^{-\delta}+ b_3|s|^{q-1}+b_4, 
	    & & \text{for a.a. } x \in \Omega,
	\end{align*}
	for all $ s \in \R$, for all $ \xi \in \RN$, with nonnegative constants $ a_i, b_j \, (i,j \in \{1, \dots, 4\})$ and fixed numbers
	\begin{align}\label{delta}
	    1< p< \infty, \qquad p\leq q\leq p^*\qquad \text{and} \qquad 0< \delta< 1.
	\end{align}
\end{enumerate}
By $p^*$ we denote the critical exponent of $p$, that is
\begin{align*}
    p^*=
    \begin{cases}
	\frac{Np}{N-p} &\text{if }p<N,\\
	\text{any }r\in(1,\infty)&\text{if }p\geq N.
    \end{cases}
\end{align*}
From \eqref{delta} we see that the case $q=p^*$ is allowed and so critical growth can occur.

We are interested in a priori bounds for weak solutions of problem \eqref{problem}. By a weak solution we mean a function $u \in W^{1,p}_0(\Omega)$ such that $u> 0$ a.e.\,in $ \Omega$ and
    \begin{align}\label{weak}
	\into \mathcal A(x, u, \nabla u) \cdot \nabla v\, dx= \into B(x, u, \nabla u) v\, dx
    \end{align}
is satisfied for all test functions $v \in W^{1,p}_0(\Omega)$. Taking into account hypotheses (H) we see that we have a well-defined weak solution. Note that for the second term on the right-hand side of (H3) the weak solution is by definition a function $u \in \Wpzero{p}$ such that $u^{-\delta}v \in \Lp{1}$ for every $v \in \Wpzero{p}$. This is a natural definition consistent with the classical definition of a weak solution.

Our main result reads as follows.

\begin{theorem}\label{theorem}
    Let $ \Omega \subset \RN$ be a bounded domain with Lipschitz boundary $\partial\Omega$ and let hypotheses (H) be satisfied. Then, any weak solution $u \in W^{1,p}_0(\Omega)$ of \eqref{problem} belongs to $ \Linf$.
\end{theorem}

The most important feature of \eqref{problem} is the presence of critical and negative exponents along with a convection term and a possibly nonhomogeneous operator which states the problem in a very general setting. A key role in the proof of Theorem \ref{theorem} is Lemma \ref{lemma1} stated in Section \ref{section_2} which transforms the weak definition of our problem (see \eqref{weak}) in the right form by using a suitable test function in order to apply powerful bootstrap arguments. 

Singular elliptic equations have been increasingly studied in the past decade and although there is no complete regularity theory for such problems, existence and multiplicity of weak solutions of \eqref{problem} have been proved in several works. We only mention, for example, Giacomoni-Schindler-Tak\'{a}\v{c} \cite{Giacomoni-Schindler-Takac-2007}, Marano-Marino-Moussaoui \cite{Marano-Marino-Moussaoui-2019}, Papageorgiou,-R\u{a}dulescu-Repov\v{s} \cite{Papageorgiou-Radulescu-Repovs-2020}, \cite{Papageorgiou-Radulescu-Repovs-2019}, Papageorgiou-Smyrlis \cite{Papageorgiou-Smyrlis-2015}, Papageorgiou-Winkert \cite{Papageorgiou-Winkert-2019}, Perera-Zhang \cite{Perera-Zhang-2005} and the references therein. For regularity results and a priori bounds for problems with convection term we refer, for example, to Ho-Sim \cite{Ho-Sim-2017}, Marino-Motrenau \cite{Marino-Motreanu-2020}, Ragusa-Tachikawa \cite{Ragusa-Tachikawa-2020}, Winkert \cite{Winkert-2010}, Winkert-Zacher \cite{Winkert-Zacher-2012} and the references therein.

The motivation of our work is the paper of Giacomoni-Schindler-Tak\'{a}\v{c} \cite{Giacomoni-Schindler-Takac-2007} (see also Giacomoni-Saoudi \cite{Giacomoni-Saoudi-2010}) who proved the boundedness of weak solutions of the singular problem
\begin{align}\label{problem2}
    \begin{aligned}
	-\Delta_p u&= \frac{\lambda}{u^\delta}+u^q \qquad && \text{in } \Omega, \\[1ex]
	u&> 0 && \text{in } \Omega, \\[1ex]
	u&= 0 && \text{on } \rand,
    \end{aligned}
\end{align}
see Lemma A.6 in \cite{Giacomoni-Schindler-Takac-2007}. We point out that our problem is much more general than those in \eqref{problem2}. Indeed, with a view to the conditions (H1) and (H2) we see that the $(p,q)$-Laplace differential operator, which is a prototype of a nonhomogeneous operator and is given by
\begin{align*}
    \Delta_pu+\Delta_qu=\divergenz\left(|\nabla u|^{p-2}\nabla u\right)+\divergenz\left(|\nabla u|^{q-2}\nabla u\right)\quad\text{for all }u \in \Wpzero{p}
\end{align*}
with $1<q<p$, fits in our setting. But also the right-hand in \eqref{problem} is more general since we allow the occurrence of a convection term, that is, the dependence on the  gradient of the solution.

\section{Proof of the Main Result}\label{section_2}

Before we give the proof of Theorem \ref{theorem} we begin with a truncation lemma which was motivated by the work of Giacomoni-Schindler-Tak\'{a}\v{c} \cite{Giacomoni-Schindler-Takac-2007}. 

Let $\Omega_{>1}:= \{x\in\Omega: u(x)>1\}$. 

\begin{lemma}\label{lemma1}
    Let the hypotheses (H) be satisfied and let $ u \in W^{1,p}_0(\Omega) $ be a weak solution of problem \eqref{problem}. Then,
    \begin{align}\label{1}
      \begin{split}
	&\int_{\Omega_{>1}} \mathcal A(x, u, \nabla u) \cdot \nabla w\,dx\le M_1 \int_{\Omega_{>1}} \l(| \nabla u |^{p \frac{q-1}{q}}+ 1 +u^{q-1}\r) w \,dx,
      \end{split}
    \end{align}
    for all nonnegative functions $ w \in W^{1,p}_0(\Omega)$ and for some $M_1>0$.
\end{lemma}

\begin{proof}
    Let $u\in \Wpzero{p}$ be a weak solution of \eqref{problem}. We take a $C^1$-cut-off function $ \eta\colon \R \to [0, 1]$ such that
    \begin{align}\label{eta}
	\eta(t)=
	\begin{cases}
	0 \quad \text{if } t\leq 0, \\
	1 \quad \text{if } t\geq 1,
	\end{cases}
	\qquad
	\eta'(t) \ge 0 \quad \text{for all }  t \in [0,1].
    \end{align}
    Moreover, for every $ \eps> 0$ we define 
    \begin{align*}
	\eta_{\eps}(t):= \eta\l(\frac{t-1}{\eps}\r).
    \end{align*}
    By means of the chain rule, we have
    \begin{align}\label{eta2}
	\eta_{\eps} \circ u \in W^{1,p}_0(\Omega) \quad \text{and} \quad \nabla(\eta_{\eps} \circ u)= (\eta'_{\eps} \circ u) \nabla u.
    \end{align}
    Moreover, taking into account the definition of $\eta_{\eps}$, it holds
    \begin{align}\label{eta5}
	\eta_{\eps} \circ u(x)= \eta\l(\frac{u(x)-1}{\eps}\r)=
	\begin{cases}
	    1 \qquad \qquad \qquad & \text{if } \frac{u(x)-1}{\eps}\geq  1, \\[1ex]
	    0 & \text{if } \frac{u(x)-1}{\eps}\leq 0, \\[1ex]
	    \eta\left(\frac{u(x)-1}{\eps}\right) & \text{otherwise}.
	\end{cases}
    \end{align}
    Now we fix a nonnegative function $w \in W^{1,p}_0(\Omega)$. Taking the test function $v=(\eta_{\eps} \circ u) w$ in the weak formulation of \eqref{problem} and applying the growth condition (H3) gives
    \begin{align}\label{eta3}
	\begin{split}
	    &\into \mathcal A(x, u, \nabla u) \cdot \nabla((\eta_{\eps} \circ u) w)\,dx \\
	    & \le \into \l(b_1 | \nabla u |^{p \frac{q-1}{q}}+ b_2 u^{-\delta}+ b_3 u^{q-1}+b_4\r) (\eta_{\eps} \circ u) w \,dx.
	\end{split}
    \end{align}
    For the left-hand side of \eqref{eta3} we get
    \begin{align}\label{eta4}
	\begin{split}
	    & \into \mathcal A(x, u, \nabla u) \cdot \nabla((\eta_{\eps} \circ u) w)\,dx \\
	    & = \into \mathcal A(x, u, \nabla u) \cdot \l[\nabla(\eta_{\eps} \circ u) w+ (\eta_{\eps} \circ u) \nabla w \r]\,dx \\
	    &  = \into  \left(\mathcal A(x, u, \nabla u) \cdot \nabla u ((\eta'_{\eps} \circ u) w)+ \mathcal A(x, u, \nabla u) \cdot \nabla w (\eta_{\eps} \circ u)\right) \,dx \\
	    & \ge \into \left(a_4 | \nabla u |^p(\eta'_{\eps} \circ u) w+ \mathcal A(x, u, \nabla u) \cdot \nabla w (\eta_{\eps} \circ u)\right)\,dx \\
	    & \ge \into \mathcal A(x, u, \nabla u) \cdot \nabla w (\eta_{\eps} \circ u) \,dx,
	\end{split}
    \end{align}
    where we used \eqref{eta}, \eqref{eta2} and \eqref{eta5}. From \eqref{eta3}-\eqref{eta4} we then derive
    \begin{align}\label{eta6}
	\begin{split}
	    &\into \mathcal A(x, u, \nabla u) \cdot \nabla w (\eta_{\eps} \circ u)\,dx\\ 
	    &\leq \into \l(b_1 | \nabla u |^{p \frac{q-1}{q}}+ b_2 u^{-\delta}+ b_3u^{q-1}+b_4 \r) (\eta_{\eps} \circ u) w \,dx.
	 \end{split}
    \end{align}
    Letting $\eps \to 0^+$, from \eqref{eta5} and \eqref{eta6}, we obtain
    \begin{align*}
	\begin{split}
	    &\int_{\Omega_{>1}} \mathcal A(x, u, \nabla u) \cdot \nabla w\, dx\\
	    & \le \int_{\Omega_{>1}} \l(b_1 | \nabla u |^{p \frac{q-1}{q}}+ b_2 u^{-\delta}+ b_3u^{q-1}+b_4\r)w\, dx \\
	    & \le \int_{\Omega_{>1}} \l(b_1 | \nabla u |^{p \frac{q-1}{q}}+ b_2+ b_3u^{q-1}+b_4\r)w\, dx \\
	    & \le M_1 \int_{\Omega_{>1}}  \l(| \nabla u |^{p \frac{q-1}{q}}+ 1+u^{q-1}\r)w\, dx,
	\end{split}
    \end{align*}
    where $M_1:= \max\{b_1, b_2+ b_4,b_3\}$. The proof is thus complete.
\end{proof}

\begin{remark}
    In a recent paper of the authors \cite{Marino-Winkert-2019} a more general coercivity condition in the form 
    \begin{align*}
	\mathcal{A}(x,s,\xi) \cdot \xi \geq a_4|\xi|^{p}-a_5|s|^{q}-a_6
    \end{align*}
    is used instead of (H2). Unfortunately, we were not able to prove Lemma \ref{lemma1} with such general estimate and used (H2) instead.
\end{remark}

\begin{proof}[Proof of Theorem \ref{theorem}]$ $

    We are going to show the proof only in the critical case $q=p^*$.
    
    {\bf Part 1: $u\in L^r(\Omega)$ for any $r<\infty$.}
    
    Let $h>0 $ and set $ u_h:= \min\{u, h\}$. Let $\kappa>0$. Testing \eqref{1} with $ w= u u_h^{\kappa p}$ results in
    \begin{align}\label{2}
	\begin{split}
	    & \int_{\Omega_{>1}} \left(\mathcal A(x, u, \nabla u) \cdot \nabla u\right) u_h^{\kappa p}\,dx\\
	    &\quad + \kappa p \int_{\Omega_{>1}} \left(\mathcal A(x, u, \nabla u) \cdot \nabla u_h\right) u_h^{\kappa p-1} u \,dx \\
	    & \le M_1 \int_{\Omega_{>1}} \l(| \nabla u |^{p \frac{p^*-1}{p^*}}+ 1+u^{p^*-1} \r) u u_h^{\kappa p}\,dx.
	\end{split}
    \end{align}
    Note that we have used $ \nabla w=  u_h^{\kappa p}\nabla u+ u \kappa p u_h^{\kappa p-1} \nabla u_h$. Now we may apply (H2) to both terms on the left-hand side. We obtain
    \begin{align}\label{3}
	\int_{\Omega_{>1}} \left(\mathcal A(x, u, \nabla u) \cdot \nabla u\right) u_h^{\kappa p} \,dx \ge a_4 \int_{\Omega_{>1}} | \nabla u |^p u_h^{\kappa p} \,dx
    \end{align}
    and
    \begin{align}\label{4}
	\begin{split}
	    &\kappa p \int_{\Omega_{>1}} \left(\mathcal A(x, u, \nabla u) \cdot \nabla u_h\right) u_h^{\kappa p-1} u \,dx	 \\
	    &= \kappa p \int_{\{x \in \Omega_{>1}: \, u(x) \le h\}} \left(\mathcal A(x, u, \nabla u) \cdot \nabla u\right) u_h^{\kappa p} \,dx \\
	    & \ge a_4 \kappa p \int_{\{x \in \Omega_{>1}: \, u(x) \le h\}} | \nabla u |^p u_h^{\kappa p} \,dx.
	\end{split}
    \end{align}
    For the right-hand side of \eqref{2} we will use Young's inequality which implies that
    \begin{align}\label{5}
	\begin{split}
	    & M_1 \int_{\Omega_{>1}} \l(| \nabla u |^{p \frac{{p^*}-1}{{p^*}}}+ 1+u^{{p^*}-1} \r) u u_h^{\kappa p} \,dx \\
	    &= M_1 \int_{\Omega_{>1}}\left[ \left(\frac{a_4}{2M_1}\right)^{\frac{{p^*}-1}{{p^*}}} | \nabla u |^{p \frac{{p^*}-1}{{p^*}}} u_h^{\kappa p \frac{{p^*}-1}{{p^*}}}\right.\\
	    &\left. \qquad  \times\left(\frac{a_4}{2M_1}\right)^{-\frac{{p^*}-1}{{p^*}}} u_h^{\kappa p \l(1- \frac{{p^*}-1}{{p^*}}\r)} u\right] \,dx \\
	    &\qquad+ M_1 \int_{\Omega_{>1}}  u u_h^{\kappa p}\,dx +M_1 \int_{\Omega_{>1}}  u^{{p^*}} u_h^{\kappa p}\,dx \\
	    & \le \frac{a_4}{2M_1} M_1 \int_{\Omega_{>1}} | \nabla u |^p u_h^{\kappa p} \,dx\\
	    &\qquad+ M_1 \l(\left(\frac{a_4}{2M_1}\right)^{-({p^*}-1)}+ 2\r) \int_{\Omega_{>1}} u^{{p^*}} u_h^{\kappa p} \,dx.
	\end{split}
    \end{align}
    Combining \eqref{3}, \eqref{4} and \eqref{5} gives
    \begin{align*}
	\begin{split}
	    &a_4 \l(\frac{1}{2} \int_{\Omega_{>1}} | \nabla u |^p u_h^{\kappa p}\, dx+ \kappa p \int_{\{x \in \Omega_{>1}: \, u(x) \le h\}} | \nabla u |^p u_h^{\kappa p} \,dx \r) \\
	    &\le M_2\int_{\Omega_{>1}} u^{p^*} u_h^{\kappa p} \,dx.
	\end{split}
    \end{align*}
    Now we can follow the treatment in \cite[from (3.7) to (3.21)]{Marino-Winkert-2019} which finally shows
    \begin{align}\label{8}
	\frac{\kappa p+1}{(\kappa+1)^p} \Vert u u_h^{\kappa} \Vert_{W^{1,p}_0(\Omega_{>1})}^p \le M_2 \int_{\Omega_{>1}} u^{p^*} u_h^{\kappa p} \,dx.
    \end{align}
    This gives
    \begin{align*}
	\Vert u \Vert_{L^{(\kappa+1)p^*}(\Omega_{>1})} \le M_3(\kappa, u)
    \end{align*}
    for any finite number $\kappa$ with a positive constant $M_3(\kappa, u)$ depending on $\kappa$ and on the solution itself. This implies that $ u \in L^r(\Omega_{>1})$ for any $ r \in (1, \infty)$. Moreover, since
    \begin{align}\label{bound}
	\begin{split}
	    \into u^r \,dx&= \int_{\Omega_{>1}} u^r\,dx+ \int_{\{x \in \Omega: \, u(x) \le 1\}} u^r\,dx \le \Big(M_3(\kappa, u)^r +1\Big) |\Omega|< \infty,
	\end{split}
    \end{align}
    where $|\Omega|$ stands for the Lebesgue measure of $\Omega$, we easily have $ u \in L^r(\Omega)$ for every $ r \in (1, \infty)$. 
    
    {\bf Part 2: $u\in \Linf$.}
    
    As done in \cite[see (3.28) until the end of Case II.1]{Marino-Winkert-2019} we can start from \eqref{8}, that is,
    \begin{align*}
	\frac{\kappa p+1}{(\kappa+1)^p} \Vert u u_h^{\kappa} \Vert_{W^{1,p}_0(\Omega_{>1})}^p \le M_2 \int_{\Omega_{>1}} u^{p^*} u_h^{\kappa p} \,dx.
    \end{align*}
    By an appropriate usage of H\"older's inequality, Fatou's lemma and the Sobolev embedding theorem we can reach that
    \begin{align*}
	\Vert u \Vert_{L^{(\kappa_n+1)p^*}(\Omega_{>1})} \le M_{4},
    \end{align*}
    where $(\kappa_n+1)p^* \to \infty $ as $ n \to \infty$. This shows that $u \in L^{\infty}(\Omega_{>1})$. The same argument as in \eqref{bound} can be applied and so we conclude that $ u \in L^{\infty}(\Omega)$.
\end{proof}


\end{document}